\documentclass[11pt]{amsart} 

\usepackage[all]{xy} 
\usepackage{amsmath} 
\usepackage{amssymb}
\usepackage{amsthm} 
\usepackage[pdfauthor={Matei Toma},
pdftitle={Vector bundles on blown-up Hopf surfaces},
pdfkeywords={compact complex surfaces, moduli spaces, Hopf surfaces, vector bundles}, pdfview={Fit}]{hyperref}

\newenvironment{pf}{\begin{proof}}{\end{proof}}
\newtheorem{Th}{Theorem}[section] 

\newtheorem{Cor}[Th]{Corollary} 
\newtheorem{Prop}[Th]{Proposition}

\newtheorem{Def}[Th]{Definition}

\def\cE{\mathcal E}

\newcommand{\cI}{\mathcal{I}}

\newcommand{\cM}{\mathcal{M}}

\newcommand{\cO}{\mathcal{O}}

\def\cM{\mathcal M } 
\def\cI{\mathcal I }

\newcommand{\C}{\mathbb{C}} 

\renewcommand{\P}{\mathbb{P}}

\newcommand{\Z}{\mathbb{Z}}

\def\id{\mbox{id}} 

\def\Ext{\mbox{Ext}}
\def\Hom{\mbox{Hom}}

\def\Pic{\mbox{Pic}}

\title{Vector bundles on blown-up Hopf surfaces}
\author[Matei Toma]{Matei Toma}
\address{Matei Toma, 
Universit\'e de Lorraine, CNRS, IECL, F-54000 Nancy, France}

\email{Matei.Toma@univ-lorraine.fr}
\urladdr{\href{https://iecl.univ-lorraine.fr/membre-iecl/toma-matei/}{https://iecl.univ-lorraine.fr/membre-iecl/toma-matei/}}

\date{December 2011}
\thanks{AMS Classification (2000): 32G13, 32J15.\\
Acknowledgement: I wish to thank the referee for his remarks which helped me to improve the exposition.}

\begin{document}
\maketitle

\begin{abstract} {We show that certain moduli spaces of vector bundles over blown-up primary Hopf surfaces admit no compact components. These are the moduli spaces used by Andrei Teleman in his work on the classification of class $VII$ surfaces.}
\end{abstract}

\section{Introduction}
Moduli spaces of holomorphic vector bundles over compact complex manifolds have been extensively studied. A case of particular interest, where much has been proved is the case when the base $X$ is a smooth projective surface \cite{HL}. With respect to an ample polarization $H$ on $X$ one considers the moduli space $\cM^s_{lf}(r,L,c_2)$ of slope-stable holomorphic vector bundles of rank $r$, determinant $L\in\Pic(X)$ and second Chern class $c_2\in H^4(X,\Z)\cong \Z$. One issue about it, of capital importance for Donaldson theory for instance, is the existence of ``modular'' compactifications. Two such compactifications of $\cM^s_{lf}(r,L,c_2)$ have been constructed: the Gieseker compactification of semi-stable torsion free sheaves and the Uhlenbeck compactification of ideal Hermite-Einstein connections. Both are projective, although the second one is constructed in the framework of gauge theory. 

In the more general case of a compact complex surface $(X,g)$ equipped with a Gauduchon metric, a degree function on $\Pic(X)$ with respect to $g$ may be defined, hence a (slope) stability notion for torsion-free coherent sheaves on $X$ \cite{LT}. Under some supplementary condition, which essentially demands that every semi-stable sheaf with the given invariants is already stable, it was shown in \cite{Tom01} that the moduli space $\cM^s(r,L,c_2)$ of stable torsion free sheaves with invariants $(r,L,c_2)$ is compact. This is a complex analytic space and provides a modular compactification of  $\cM^s_{lf}(r,L,c_2)$ which may be thought as an analogue of the Gieseker compactification. 
In general however $\cM^s(r,L,c_2)$ is not compact and this is related to the fact 
that the Uhlenbeck compactification, which always exists, may admit no compatible complex space structure.

The study of the moduli spaces of holomorphic vector bundles allowed  Andrei Teleman in  \cite{Tel05}, \cite{Tel10}  to make  a breakthrough towards a complete classification of non-K\"ahlerian compact complex surfaces. One of the main issues of his study was the lack of compactness of $\cM^s(2,K_X,0)$ in his situation. 
The purpose of this note is to show the non-compactness of $\cM^s(2,K_X,0)$  in the case when $X$ is a blown-up primary Hopf surface. The case of blown-up Hopf surfaces is particularly important in  light of Teleman's main result from  \cite{Tel10}. Indeed, it follows from loc. cit. that all minimal surfaces $X$ with $b_1(X)=1$, $b_2(X)=2$ are deformations of blown-up Hopf surfaces, still their complete classification is not yet available.  In fact our non-compactness result was used by Sch\"obel in \cite{Sch08} via a deformation argument in order to describe $\cM^s_{lf}(2,K_X,0)$ when $X$ is a minimal surface with $b_1(X)=1$, $b_2(X)=1$.  This type of argument is an indication that the part of Teleman's program dealing with non-compactness phenomena should work also for $b_2(X)>2$.

\section{Families of extensions}
We consider a primary Hopf surface $X$ and $\hat X$ to be $X$ blown-up at
 points $x_1, ..., x_n \in X$, $n\geq 1$. We shall denote by
 $\pi :\hat X \rightarrow    X$ the blowing down morphism and by $D_1+...+D_n$
 the exceptional divisor. Vector bundles will not
 be distinguished from their sheaves of holomorphic sections. 
One has $\Pic (X) \cong \C ^*$ and $\Pic(\hat X)$
 is the product of $\pi ^*(\Pic(X))$ with the free abelian group generated
 by   $\cO (D_1)$, ..., $\cO (D_n)$. Therefore any invertible sheaf on 
 $\hat X$ has the form $\pi ^*(L)(D)$ for some $L \in \Pic (X)$ and
 $D= \sum a_i D_i$, $a_i\in\Z$. We shall write shortly $L(D)=\pi ^*(L)(D)$ for it.
In particular $K_{\hat X}=K_X(D_1+...+D_n)$. 
For any rank two coherent sheaf $E$ on a compact complex surface $X$
one defines the discriminant:
\[
 \Delta(E) := \frac{1}{2}\left(c_2(E)-\frac{1}{4}c_1(E)^2\right).
\]
Recall that for a torsion free coherent sheaf $E$ on a non-algebraic surface one always has
\[
 \Delta(E)\ge 0.
\]

\begin{Prop}\label{tip}
Let $\hat E$ be a torsion free sheaf of rank two on $\hat X$ with $\det(\hat E)=K_{\hat X}$ and $c_2(\hat E)=0$. Then:
\begin{enumerate}
\item $\hat E$ is locally free and $\Delta(\hat E)=\frac{n}{8}$.
\item There are no torsion free sheaves $F$ of rank two on $\hat X$ with $\det(F)=K_{\hat X}$ and $c_2(F)<0$.
\item $\hat E$ is the central term of an extension of the form
$$0 \rightarrow K_{\hat X} (-D)\otimes L^{-1} \rightarrow \hat E \rightarrow  L(D) \rightarrow 0,$$
where $L\in \Pic(X)$ and $D= \sum a_i D_i$ with $a_i\in \{ 0,1\}$.
\end{enumerate}

\end{Prop}
\begin{proof}
It is clear that  $\det(\pi _*(\hat E))=K_X$.
One has $\chi(\hat E)=\chi(\pi _*(\hat E))-\chi(R^1\pi_*(\hat E))\le \chi(\pi _*(\hat E))$. By Riemann-Roch we compute $\chi(\hat E)$ and $\chi(\pi _*(\hat E))$ as follows
$$ \chi(\hat E)= 2(\chi(\cO_{\hat X})-\frac{1}{4}c_1(\hat E)c_1(K_{\hat X})+\frac{1}{8}c_1(K_{\hat X})^2-\Delta(\hat E))=0$$
$$ \chi(\pi _*(\hat E))= 2(\chi(\cO_{ X})-\frac{1}{4}c_1(\pi _*(\hat E))c_1(K_{X})+\frac{1}{8}c_1(K_{X})^2-\Delta(\pi _*(\hat E)))=-2\Delta(\pi _*(\hat E)).$$ 

Combining this with $\chi(\hat E)=\chi(\pi _*(\hat E))-\chi(R^1\pi_*(\hat E))\le \chi(\pi _*(\hat E))$ and with the inequality $\Delta(\pi _*(\hat E))\ge 0$ we get
$$\chi(\hat E)=\chi(\pi _*(\hat E))=\Delta(\pi _*(\hat E))=0$$
and
$$R^1\pi_*(\hat E)=0.$$
In particular $\pi _*(\hat E)$ has to be locally free since otherwise we would have 
$\Delta((\pi _*(\hat E))^{\vee\vee})<0$.

We also get an exact sequence on $X$:
$$0\to \pi_*(\hat E)\to \pi _*(\hat E^{\vee\vee})\to \pi _*(\hat E^{\vee\vee}/ \hat E)\to 0$$
showing that $\hat E^{\vee\vee}/ \hat E=0$ and $\hat E$ is locally free.

If $F$ were a torsion free sheaf of rank two on $\hat X$ with $\det(F)=K_{\hat X}$ and $c_2(F)<0$, then the
same computations as before would give $\Delta(\pi_*(F))<0$, which is absurd.

Now $c_2(\pi _*(\hat E))=0$, hence $\pi _*(\hat E)$ cannot be stable with respect to any Gauduchon metric 
on $X$. Otherwise there would exist
some irreducible $SU(2)$-valued representation of the fundamental group of $X$ which is 
cyclic infinite, cf. \cite{Pla95}. But this is not the case. Thus $\pi _*(\hat E)$ admits some coherent subsheaf of rank one which implies that $\hat E$ also admits one. This leads to the existence of an exact sequence of the form:
$$0\to L_1\to \hat E\to L_2\otimes \cI_Z\to 0,$$
where $L_1$, $L_2$ are line bundles on $\hat X$ and $Z$ is a locally complete intersection subspace of codimension two of $X$. Now if $Z$ were not empty the vector bundle $F:=L_1\oplus L_2$ would have the same determinant as $\hat E$ but a strictly lower second Chern class, which would contradict our second assertion. Thus $Z$ must be empty.

We may write now $L_2=L(D)$ for some $L \in \Pic(X)$ and $D=\sum a_i D_i$, $a_i\in \Z$. Then $L_1=K_{\hat X} (-D)\otimes L^{-1}$ and $0=c_2(\hat E)=c_1(K_{\hat X} (-D))c_1(\cO(D))=-\sum a_i(a_i-1)$. The last sum vanishes if and only if each $a_i$ is zero or one, which proves our last claim.
\end{proof}

We are thus interested in extensions of the type 
 $$0 \rightarrow K_{\hat X} (-D)\otimes L^{-1} \rightarrow \hat E \rightarrow  L(D) \rightarrow 0
$$
as in Proposition \ref{tip}.
 
\begin{Prop}\label{extensions} 
Let $L$ be an element of $ \Pic(X)$ and  $D= \sum a_i D_i$ with $a_i\in \{ 0,1\}$. Then the dimension of the projective space of non-trivial extensions of $ L(D)$ by 
 $ K_{\hat X}(-D)\otimes L^{-1}$ is $-D^2 + \epsilon$,  where 
 $\epsilon \ge 0$ when $h^0(X;L^{\otimes 2})\neq 0$ or when $D=0$ and 
 $h^0(X;K_X \otimes L^{\otimes -2})\neq 0$. Otherwise $\epsilon =-1$.
In particular the dimensions of the above spaces of extensions do not exceed $n-1$, unless  $h^0(X;L^{\otimes 2})\neq 0$ or $h^0(X;L^{\otimes 2})\neq 0$. 
If $X$ is not an elliptic surface, the maximal dimension for such a space is $n$ and it is attained
precisely when  $D= D_1+...+D_n$ and $h^0(L^{\otimes 2})\neq 0$.
\end{Prop}

\begin{pf}
 The dimension of the projective space of non-trivial extensions of $ L(D)$ by 
 $ K_{\hat X}(-D)\otimes L^{-1}$  is computed as follows
$$\dim \Ext^1(\hat X;L(D),K_{\hat X}(-D)\otimes L^{-1})-1=h^1(\hat X;K_{\hat X}(-2D)\otimes L^{-2})-1=$$
$$h^1(\hat X;L^2(2D))-1=-\chi(L^2(2D))+h^0(\hat X;L^2(2D))+h^0(\hat X;K_{\hat X}(-2D)\otimes L^{-2})-1=$$
$$-2D^2+Dc_1(K_{\hat X})+h^0(X;L^2)+h^0(\hat X;K_{\hat X}(-2D)\otimes L^{-2})-1=$$
$$-D^2+h^0(X;L^2)+h^0(\hat X;K_{\hat X}(-2D)\otimes L^{-2})-1,$$
whence our claim.
\end{pf}

\section{The main result}

\begin{Def}
A {\bf coarse family} of vector bundles over an analytic space $X$ parameterized by an analytic space $T$
is a family of isomorphy classes $([E_t])_{t\in T}$ of holomorphic vector bundles over $X$ together with a covering of 
$T$ by open subsets $T_i$ such that over each $T_i \times X$ there exists a holomorphic vector bundle
$\cE_i$ with $\cE_i|_{\{ t\} \times X}\cong E_t$ for all $t\in T_i$.
 We say that the coarse family is {\bf effective} if the restricted families over the $T_i$-s 
are effective in the usual sense. \end{Def}

\begin{Th}
Let $\hat X$ be the blow-up of a primary Hopf surface $X$ at $n$ points and fix a Gauduchon metric
$\hat g$ on $\hat X$. Let $T$ be a compact irreducible analytic space parameterizing a coarse
family of semi-stable  rank $2$ vector bundles with determinant $K_{\hat X}$ and vanishing second Chern class
on $(\hat X, \hat g)$. Suppose that a non-empty open part of $T$ 
effectively parameterizes simple vector bundles.
Then $dim (T)<n$.  
\end{Th}

{\it Proof.} 
Let $\hat E_t$ be the vector bundles of the given coarse family. 
Under the above assumptions and notations
we first prove that the vector bundles $E_t:=\pi _*(\hat E_t)$ are also organized in a coarse family
over $X$. For this it is enough to check that $(id_{T_i}\times \pi)_*(\hat {\cE _i})$ are locally free
over $T_i \times X$. This statement is local around a point $(t, x_j)\in T_i \times X$, 
so suppose for the moment that $T=T_i$ 
and $X$ are just small neighbourhoods of the points $t$ and $x=x_j$.
We may also assume that $T$ is irreducible and
non-singular.
We view $\hat X$ as the zero set of a section of the pullback of 
$\cO _{\P ^1}(1)$ to $T \times X \times \P ^1$ 
and consider some locally free extension 
$\hat {\cE} '$ of $\hat {\cE}$   to $T \times X \times \P ^1$, cf. \cite{Bu00} Lemma 2.2. 
Denote by $p:T \times X \times \P ^1 \rightarrow T \times X$ the projection
and by $\iota $ the embedding of $T \times \hat X$ into $T \times X \times \P ^1$.  
From  Proposition \ref{tip} it follows that the restriction of $\hat E_t$ to an exceptional divisor $D_i$ is isomorphic to
 $\cO \oplus \cO (-1)$. By semi-continuity the splitting type
of  $\hat {\cE} '$ over each vertical line will remain $\cO \oplus \cO (-1)$. Over 
$T \times X \times \P ^1$ we have an exact sequence:
$$0 \rightarrow \hat {\cE}'(-1) \rightarrow \hat {\cE}' \rightarrow \iota _*(\hat {\cE}) \rightarrow 0$$
whose push-forward through $p$ gives:
$$ 0 \rightarrow p_*(\hat {\cE}') \rightarrow (\id_T\times\pi) _*(\hat {\cE}) \rightarrow R^1p_*(\hat {\cE}'(-1)) \rightarrow 0.$$
Thus $(\id_T\times\pi) _*(\hat {\cE}) $ will be locally free
as an extension of locally free terms.

Consider now the rank $2$ vector bundles $E_t$ over $X$. In the proof of Proposition \ref{tip} we have seen that they cannot be stable   with respect to any Gauduchon metric on $X$. Thus each $E_t$ allows  some destabilizing subsheaf
in $\Pic(X)$. The compactness of $T$ will allow us to 
find such a destabilizing subsheaf in an
uniform way:

Recall first that $\Pic(X)\cong \C ^*$ and that 
the degree map on $\Pic(X)$ with respect to a Gauduchon  
metric on $X$ corresponds to a positive constant times the logarithm of the absolute value defined
on $\C ^*$, hence any two degree maps are proportional, \cite{LT} 1.3.15. In fact we shall use the degree map induced by $\hat g$ on $\Pic(X)$ via the embedding $\Pic(X)\to\Pic(\hat X)$.
We  denote by $L_w$ the line bundle on $X$
which corresponds to $w \in \C ^*$.  The analytic subspace 
$Z:=\{ (t,w)\in T\times \C ^* \ | \ \Hom (L_w, E_t)\neq 0 \}$
is proper over $\C ^*$ but doesn't cover $\C ^*$, otherwise we would get
subbundles $L_w$ of arbitrarily high degree whose pullbacks to $\hat X$ would contradict the
semi-stability of some $\hat E_t$. Thus the projection of $Z$ to $\C^*$ is an analytic subspace of $\C^*$ contained in a punctured closed disc of $\C$.
It follows the existence of a $w\in \C ^*$ with $2deg (L_w) \geq deg(K_X)$ and 
$Hom (L_w, E_t)\neq 0 $ for all $t\in T$.

The composition $L_w \rightarrow \pi ^*(E_t) \rightarrow \hat E_t$ 
will factorize through a line bundle $K_{\hat X} (-D)\otimes L^{-1}$ giving an extension  
$$0 \rightarrow K_{\hat X} (-D)\otimes L^{-1} \rightarrow \hat E _t 
\rightarrow  L(D) \rightarrow 0$$
as in Proposition \ref{tip}. Therefore there will be a nontrivial morphism
$L_w \rightarrow K_{ X} \otimes L^{-1} $ on $X$, hence
$2deg(L) \leq deg(K_X) <0$. From Proposition \ref{extensions} it follows that dimension of the projective space of extensions 
of the above  form 
is at most $n-1$ in this case. Note also that the possible line bundles $L$ appearing in such extensions run through a countable subset of $\Pic(X)$.

Now  each $\hat E_t$ in our family is the middle term of such an extension. Using the effectivity hypothesis on our family and the universal property of the moduli space of simple sheaves, we see
 that the open part of $T$ which effectively parameterizes simple vector bundles
must be covered by the images of an at most countable number of 
spaces of extensions of dimensions 
less than $n$. This entails  $dim (T)<n$.
\qed

The following corollary is an immediate consequence of our considerations and of the usual dimension estimate for $\cM^s(r,L,c_2)$, \cite{Tom01}.

\begin{Cor}
The moduli space $\cM^s(2,K_{\hat X},0)$ of stable torsion free  sheaves of rank two, determinant $K_{\hat X}$ and vanishing second Chern class on $(\hat X, \hat g)$ coincides with $\cM^s_{lf}(2,K_{\hat X},0)$, has dimension at least $n$ and contains no compact component.
\end{Cor}

Remark that there are choices of Gauduchon metrics $ \hat g$ such that $\cM^s(2,K_{\hat X},0)$ is not empty. Indeed it was shown in \cite{Tel06} that there even exist metrics $ \hat g$  such that the central term of the unique ``canonical extension''
$$0 \rightarrow K_{\hat X} \rightarrow \hat E \rightarrow  \cO_{\hat X} \rightarrow 0$$
is stable.


 \hrule \medskip
\par\noindent

\end{document}